\documentclass[12pt,english]{article}

\usepackage{amssymb}
\usepackage{polski}
\usepackage{ifthen}
\selecthyphenation{polish}
\mathchardef\ogon="012C%
\newcommand{\as}{a\kern-0.22em\lower.40ex\hbox{$_{\ogon}$}}
\newcommand{\As}{A\kern-0.22em\lower.40ex\hbox{$_{\ogon}$}}
\newcommand{\es}{e\kern-0.24em\lower.40ex\hbox{$_{\ogon}$}}
\newcommand{\Es}{E\kern-0.22em\lower.40ex\hbox{$_{\ogon}$}}
\makeatletter
\usepackage[a4paper, left=2.5cm, right=2.5cm, top=2.5cm, bottom=3.5cm, headsep=1.2cm]{geometry}

\newtheorem{theorem}{Theorem}[section]

\newtheorem{example}[theorem]{Example}

\newtheorem{lemma}[theorem]{Lemma}

\newtheorem{proposition}[theorem]{Proposition}
\newtheorem{remark}[theorem]{Remark}

\newenvironment{proof}[1][Proof]{\noindent\textbf{#1.} }{\ \rule{0.6em}{0.6em}}
\def\qed{\hbox to 0pt{}\hfill$\rlap{$\sqcap$}\sqcup$}
\makeatother
\usepackage{amssymb}
\usepackage{amsmath}
\usepackage{amsfonts}
\usepackage{color}
\usepackage{xcolor}
\usepackage{graphicx}
\usepackage{graphics}
\numberwithin{equation}{section}
\usepackage{babel}
\title{On the existence of directional derivatives for strongly cone-paraconvex mappings }
\author{Ewa Bednarczuk
\thanks{Systems Research Institute PAS, ul. Newelska 
	6, 01-447 Warszawa, Poland, Ewa.Bednarczuk@ibspan.waw.pl, 
	Warsaw University of Technology
	 Koszykowa  75,  00-662 Warsaw, Poland, E.Bednarczuk@mini.pw.edu.pl
}, Krzysztof Le\'sniewski
\thanks{Systems Research Institute PAS, ul. Newelska 
	6, 01-447 Warszawa, Poland, Warsaw University of Technology, Faculty of Mathematics and Information Science, ul. Koszykowa 75,
 00-662 Warsaw, Poland,  k.lesniewski@mini.pw.edu.pl
} }
\date{}

\begin{document}
	\maketitle
	\begin{flushright}
	Dedicated to Michel Th\'era on the occasion of his $70^{th}$ birthday.
	\end{flushright}

	\begin{abstract}
We investigate the existence of directional derivatives for strongly cone-paraconvex mappings. Our result is an extension of the classical Valadier result on the existence of  the directional derivative for cone convex mappings with values in weakly sequentially Banach space.
\end{abstract}
{\bf 2010 Mathematics Subject Classification}. {Primary 	49J50; Secondary 	52A41}

{\bf Keywords: }directional derivative, normal cones, strongly paraconvex mappings, cone convex mappings, weakly sequentially complete Banach spaces

		\section{Introduction}
	The concepts of approximate convexity for extended real-valued functions 
	include among others, $\gamma-$paraconvexity \cite{rolewicz1, rolewicz2},
	$\gamma$-semiconcavity \cite{cannarsa},  $\alpha$- paraconvexity,
	strong $\alpha$-paraconvexity \cite{rolewicz3},  semiconcavity \cite{cannarsa}, 
	approximate convexity \cite{thera}. Relations between these concepts were investigated by Rolewicz \cite{rolewicz1, rolewicz2, rolewicz3}, Daniilidis, Georgiev \cite{georgiev},  Tabor, Tabor \cite{tabor}. These concepts were used, e.g. in \cite{cannarsa} to investigate Hamilton-Jacobi equation. In a series of papers  \cite{rolewicz1, rolewicz2, rolewicz3} Rolewicz investigated Gateaux and
	 Fr\'echet differentiability of strongly $\alpha$-paraconvex, generalizing in this way the Mazur theorem (1933).
	 
	 Generalization of the above concepts to vector-valued mappings with values in general vector space $Y$ were given by Vesel\'y, Zajicek \cite{vesely_zajicek1,vesely_zajicek2,vesely_zajicek3,vesely_zajicek4}, Valadier \cite{valadier}, Rolewicz \cite{rolewicz4}.
	 In the paper \cite{rolewicz4} Rolewicz defined vector-valued strongly
	 $\alpha$-$k$ paraconvex mappings and investigated their Gateaux and Fr\'echet
	 differentiability, where $k\in K$ and $K$ is a closed convex cone in a normed vector space $Y$.

		Let $\alpha:\mathbb{R}_{+}\rightarrow\mathbb{R}_{+}$ be a nondecreasing function satisfying the condition
		$$
		\lim_{t\rightarrow 0^{+}}\frac{\alpha(t)}{t}=0.
	$$

			Let $X$ be a normed space and let $k\in K$. The mapping $F:X\rightarrow Y$ is {\em strongly $\alpha$-$k$  paraconvex} 
			on a convex subset $A$ of $X$ if  there exists a constant $C>0$ such that for every
			$x_{1},x_{2}\in A$ and every $\lambda\in[0,1]$
			\begin{equation}
			\label{def_para}
			F(\lambda x_{1}+(1-\lambda)x_{2})\le_{K} \lambda F(x_{1})+(1-\lambda) F(x_{2})+C\min\{\lambda, 1-\lambda\}\alpha(\|x_{1}-x_{2}\|)k,
			\end{equation}
			where $
			x\le_K y \iff y-x \in K.$ In the sequel we use the notation $\le$ if the cone $K$ is clear from the context.
		
		The mapping $F:X\rightarrow Y$ is strongly $\alpha$-$K$ paraconvex
		on a convex subset $A$ of $X$ if for every $k\in K$ there exists a constant $C>0$ such that for every
		$x_{1},x_{2}\in A$ and every $\lambda\in[0,1]$
		\begin{equation}
		\label{def_para_1}
		F(\lambda x_{1}+(1-\lambda)x_{2})\le_{K} \lambda F(x_{1})+(1-\lambda) F(x_{2})+C\min\{\lambda, 1-\lambda\}\alpha(\|x_{1}-x_{2}\|)k.
		\end{equation}
		
		A strongly $\alpha(\cdot)$-$K$ paraconvex mapping $F$ is called strongly cone-paraconvex if cone $K$ and the function $\alpha $ are clear from the context. 
		Since for every $\lambda\in[0,1]$
		$$
		\lambda(1-\lambda)\le\min\{\lambda, 1-\lambda\}\le 2	\lambda(1-\lambda)
		$$
		condition \eqref{def_para} 
		% and \eqref{def_para_1} 
			can be equivalently rewritten
		as
		\begin{equation}
		\label{def_para_2}
		F(\lambda x_{1}+(1-\lambda)x_{2})\le_{K} \lambda F(x_{1})+(1-\lambda) F(x_{2})+2C\lambda(1-\lambda)\alpha(\|x_{1}-x_{2}\|)k.
		\end{equation}
		%where $\tilde{C}$ is generally different from $C$.
		
			Strong cone-paraconvexity generalizes the  cone convexity.
			The mapping $F:X\rightarrow Y$ is $K$-convex
			on a convex subset $A$ of $X$ if  for every
			$x_{1},x_{2}\in A$ and every $\lambda\in[0,1]$ 
			\begin{equation}
			\label{def_conv_2}
			F(\lambda x_{1}+(1-\lambda)x_{2})\le_{K} \lambda F(x_{1})+(1-\lambda) F(x_{2}).
			\end{equation}

			In the present paper we investigate  the existence of directional derivatives for strongly cone-paraconvex mappings. Our main result (Theorem \ref{main}) is a generalization of the theorem of Valadier \cite{valadier}  concerning directional differentiability of cone convex mappings. 
		%\begin{comment}
	%	Our aim is to prove the following theorem.
	%\begin{theorem}
%		Let $X$ be a normed space. Let $Y$ be a weakly sequentially complete Banach space
%		ordered by a closed convex normal cone $K$. 
%		Let $F:X\rightarrow Y$  be strongly $\alpha(\cdot)$-$k_{0}$  paraconvex  on a convex set $A\subset X$ with constant $C\ge 0$,  $k_{0}\in K\setminus\{0\}$. Then the directional derivative
%		$$
%		F'(x_{0};h):=\lim_{t\rightarrow 0^{+}}\frac{F(x_{0}+th)-F(x_{0})}{t}
	%	$$
	%	of $F$  exists at all $x_{0}\in A$ in any direction $h\in X$, 
	%$\|h\|=1$ 
	%such that $x_{0}+th\in A$ for all $t$ sufficiently small.
	%	\end{theorem}
%		
%\end{comment}
	\section{Preliminary facts}

Let $Y^{*}$ be the dual space of $Y$ and $K^{*}\subset Y^{*}$
be the positive dual cone to $K$,
$$
K^{*}:=\{y^{*}\in Y^{*}\ |\ y^{*}(y)\ge 0\ \forall\\ \ y\in K\}.
$$

Clearly, if $F$ is  a strongly $\alpha(\cdot)$-$k$  paraconvex mapping with constant $C>0$,
then  for every $y^{*}\in K^*$ the function $y^{*}\circ F$ is a strongly  $\alpha(\cdot)$-paraconvex function with the  constant $C\cdot y^{*}(k)$. 
%On the other hand, if there exists a
%strongly  $\alpha(\cdot)$-paraconvex function $g:X\rightarrow\mathbb{R}$
%with constant $C$
%such that  for every $y^{*}\in K^*$ the function $y^{*}\circ F +g$ is convex, then $y^{*}\circ F$ is a strongly $\alpha$-paraconvex function with the same constant $C>0$. Indeed, assuming that
%$y^{*}\circ F +g$ is convex we have
%$$
%y^{*}\circ F(\lambda x_{1}+(1-\lambda) x_{2}) +g(\lambda x_{1}+(1-\lambda) x_{2})\le \lambda y^{*}\circ F(x_{1})+\lambda g(x_{1})+(1-\lambda)y^{*}\circ F(x_{2})+(1-\lambda) g(x_{2}).
%$$
%Consequently, for every $y^{*}\in K^{*}$,
%\begin{equation}
%\label{eq_zajicek}
%\begin{array}[t]{l}
%y^{*}\circ F(\lambda x_{1}+(1-\lambda) x_{2})\\
%\le \lambda y^{*}\circ F(x_{1})+\lambda g(x_{1})+(1-\lambda)y^{*}\circ F(x_{2})+(1-\lambda) g(x_{2})- g(\lambda x_{1}+(1-\lambda) x_{2})\\
%\le \lambda y^{*}\circ F(x_{1})+(1-\lambda)y^{*}\circ F(x_{2})+C\alpha(\|x_{1}-x_{2}\|).
%\end{array}
%\end{equation}

In a  normed space $Y$ a cone $K$ is normal (see \cite{valadier}) if there is a number $C>0$ such that 
$$
0\le_K x\le_K y \Rightarrow \|x\|\le C \|y\| \mbox{ for all } x, y \in Y.
$$
Every normal cone is pointed i.e. $K\cap (-K)=\{0\}.$

In \cite{vesely_zajicek1}, Vesel\`y  and Zaji$\check{c}$ek introduced the  concept of d.c. (delta-convex) mappings acting between Banach spaces $X$ and $Y$. A mapping $F:X\rightarrow Y$ is d.c. if there exists
a continuous convex function $g:X\rightarrow \mathbb{R}$  such that
for every $y^{*}\in Y^*$ the function $y^{*}\circ F+g$ is a d.c. function, i.e., it is representable as a difference of two convex functions.

	%The sum of two $k$-$\alpha$ strongly paraconvex mapping is also a 
	%$k$-$\alpha$ strongly paraconvex mapping.
	%	
	
		%\color{red}

	According to \cite{vesely_zajicek3},
		 $F$ is order d.c. if $F$ is representable as a difference of two  cone convex mappings on $A$.  Consequently, if the cone $K$ is normal, then $F$ is also weakly order d.c.
		 
		 Moreover, if the range space $Y$ of an order d.c. mapping $F$ is  ordered  by a  well-based cone $K$ (and this is true for $L_1(\mu))$, it is easy to show (see Proposition 4.1 \cite{vesely_zajicek3}) that the mapping is then d.c.

	In the example below we show that any strongly $\|\cdot\|^2$-$k_{0}$-paraconvex mapping is order d.c.
	\begin{example}	
	Let $X$ be a Hilbert space.
	A mapping  $F:X \rightarrow Y$ is  strongly $\|\cdot\|^2$-$k_{0}$-paraconvex with  constant $C\ge0$ on a convex set $A$ if and only if the mapping $F+  C  \|\cdot\|^2k_0$ is $K$-convex on  $A$.
	Indeed, let $x_1, x_2 \in X$.
	Since
	\begin{equation}
	\label{zay1}
	\lambda \|x_1\|^2 + (1-\lambda) \|x_2\|^2 - \|\lambda x_1 + (1-\lambda x_2)\|^2= \lambda(1-\lambda) \|x_1-x_2\|^2
	\end{equation}
	and
	\begin{align*}
	F(\lambda x_{1}+(1-\lambda)x_{2})\le_{K} \lambda F(x_{1})+(1-\lambda) F(x_{2})+{C}\lambda(1-\lambda)\|x_{1}-x_{2}\|^2k_0\\
\end{align*}
	 %which from \eqref{zay1}  is equivalent to
	 we have
	 \begin{align*}
	F(\lambda x_{1}+(1-\lambda)x_{2}) +C\|\lambda x_1 + (1-\lambda) x_2)\|^2k_0 & \le_{K}\\ \lambda F(x_{1})+(1-\lambda) F(x_{2})+C\lambda \|x_1\|^2k_0 + (1-\lambda) \|x_2\|^2k_0 .
	\end{align*}
	The mapping $F(\cdot)= F(\cdot) + C\|\cdot\|^2k_0$ is clearly order d.c. Furthermore, if $K$ is well based ($\exists y^*\in Y^*$ such that $y^*(k)\ge \|k\|$ for any $k\in K$), then $F$ is d.c.
		\end{example}

	%In our paper we prove the existence of directional derivative for strongly $\alpha(\cdot)$-$K$  paraconvex mappings.
	For d.c. mappings we have the following result on the existence of directional derivative.
	\begin{theorem}[Proposition 3.1 of \cite{vesely_zajicek1}]
	Let $X$ be a normed linear space and let $Y$ be a Banach space. Let $G\subset X$ be an open convex set an let $F: G \rightarrow Y$ be a d.c. mapping. Then the directional derivative $F'(x_0,h)$ exists whenever $x_0\in G$ and $h\in X.$
	\end{theorem}
	Let us observe that if function $\alpha(\cdot)$ is not convex, then we cannot expect a strongly $\alpha(\cdot)$-$k_0$ paraconvex mapping $F$ to be d.c. 	
	
	\section{Monotonicity of difference quotients}
	
	Let $X$ be a normed space. Let $Y$ be a topological vector spaces and let $K\subset Y$ be a closed convex pointed cone.
	%A point  $\bar{x}$ is in relative interior of cone $K$ i.e. $\bar{x} \in Int_rK$ if there is some $\gamma>0$ such that 
	%$\bar{x} + \gamma \mathbb{B} \cap \mbox{aff}K\subset K.$

	%The mapping $F:X\rightarrow Y$ is $K$-paraconvex
	%on a convex subset $A$ of $X$ if for every $k\in K$ there exists a constant $C>0$ such that for every
	%$x_{1},x_{2}\in A$ and every $\lambda\in[0,1]$
	%\begin{equation}
	%\label{def_para_0}
	%F(\lambda x_{1}+(1-\lambda)x_{2})\le_{K} \lambda F(x_{1})+(1-\lambda) F(x_{2})+C\min\{\lambda, 1-\lambda\}\alpha(\|x_{1}-x_{2}\|)k
	%\end{equation}

	For $K$-convex mappings, the difference quotient is nondecreasing in the sense that
	$$
	\phi(t_{1})-\phi(t_{2}):=\frac{F(x_0+t_1h)-F(x_0)}{t_1}-\frac{F(x_0+t_2h)-F(x_0)}{t_2}\in K \ \mbox{for } t_1 \ge t_2.
	$$
	For strongly $\alpha(\cdot)$-$K$  paraconvex and  strongly $\alpha(\cdot)$-$k_{0}$  paraconvex mappings, the  difference quotient may not be nondecreasing.
	\begin{example}
	Let $Y=\mathbb{R}$, $K=\mathbb{R}_+$, $\alpha(x)=x^2$ and let $F(x)=-x^2.$
	The mapping $F$ is strongly $\alpha(\cdot)$-$K$-paraconvex. Observe that for any $x_{1}, x_{2}\in \mathbb{R}$
	we $t(x^2_1+x_2^2)-2t(x_1x_2)\le 0$ if and only if $t\le0.$ Hence,
	for $t=-\lambda^2+\lambda-1\le0$ we have
	\begin{align*}
	(-\lambda^2+\lambda-1)(x_1^2+x_2^2)-2x_1x_2(-\lambda^2+\lambda-1)\le 0\\
	x_1^2(-\lambda^2+\lambda-1)+x_1x_2(-2\lambda(1-\lambda)+2)+x_2^2(-(1-\lambda)^2+1-\lambda-1)\le 0\\
	-(\lambda x_1 +(1-\lambda)x_2)^2\le -\lambda x_1^2-(1-\lambda)x_2^2+(x_1-x_2)^2\\
	F(\lambda x_1 + (1-\lambda)x_2)\le \lambda F(x_1) + (1-\lambda) F(x_2) + (x_1-x_2)^2.
	\end{align*}
	Last inequality and Proposition 2.1 from \cite{jourani} give us paraconvexity of mapping $F$.
	
	Let $x_0=0, h=1.$ Difference quotient $\phi(t)=\frac{F(x_0+th)-F(x_0)}{t}$ is decreasing. Indeed, for $t_1\le t_2$ we have
	$\phi(t_1)=-t_1$ and $\phi(t_2)=-t_2.$
	\end{example}

The following two propositions  are basic  tools for the proof of
the main result in the next section.
In the proposition below we investigate the monotonicity properties of the $\alpha(\cdot)$-difference quotients for  strongly $\alpha(\cdot)$-$k$ paraconvex mappings.
\begin{proposition}
\label{mono}
	Let $X$ be a normed space and let $Y$ be a vector spaces and ordered by a convex pointed cone $K$.
	Let $F:X\rightarrow Y$ be   strongly $\alpha(\cdot)$-$k_{0}$  paraconvex  on a convex set $A\subset X$ with constant $C\ge 0$,  $k_{0}\in K\setminus\{0\}$. For  any  $x_{0}\in A$ and any $ h\in X$, 
	$\|h\|=1$ 
	such that $x_{0}+th\in A$ for all $t$ sufficiently small,
	 the  $\alpha(\cdot)$-difference quotient mapping $\phi:\mathbb{R}\rightarrow Y$ defined as
	\begin{equation}
	\label{phi0}
	\phi(t):=\frac{F(x_{0}+th)-F(x_{0}+t_{0}h)}{t-t_{0}}+
	C\frac{\alpha(t-t_{0})}{t-t_{0}}k_{0} \ \text{for}\ t_{0}<t, 
	\end{equation}
	where $t_{0}\in\mathbb{R}$	
	is  $\alpha(\cdot)$-nondecreasing in the sense that 
	\begin{equation}
	\label{nier}
	\phi(t)-\phi(t_{1})+C\frac{\alpha(t_1-t_{0})}{t_{1}-t_{0}}k_0\in K\ \ \text{for}\ \ t_{0}<t_{1}<t.
	\end{equation}
	
	\end{proposition}
\begin{proof}
	Take any $t_{0}<t_{1}<t$.
	 We have $0<\lambda:=\frac{t_{1}-t_{0}}{t-t_{0}}<1$ and
	$$
	x_{0}+t_{1}h
	=\lambda(x_{0}+th)+(1-\lambda)(x_{0}+t_{0}h).
	$$
	Let  $k_{0}\in K\setminus \{0\}$. Since $F$ is  strongly $\alpha(\cdot)$-$k_{0}$  paraconvex with  constant $C\ge0$  we have
	\begin{equation}
	F(x_{0}+t_{1}h)
		\begin{array}[t]{l}\le_{K}\lambda F(x_{0}+th)+(1-\lambda)F(x_{0}+t_{0}h)\\
	+C\min\{\lambda, 1-\lambda\}\alpha(t-t_{0})k_{0}.
	\end{array}
	\end{equation}
	Hence,
	$$
	0\le_{K}\lambda[F(x_{0}+th)-F(x_{0}+t_{0}h)]-[F(x_{0}+t_{1}h)-F(x_{0}+t_{0}h)]+C\min\{\lambda, 1-\lambda\}\alpha(t-t_{0})k_{0}
	$$
	i.e.
	\begin{equation}
	\label{formula_1}
	[\frac{F(x_{0}+th)-F(x_{0}+t_{0}h)}{t-t_{0}}]-
		[\frac{F(x_{0}+t_{1}h)-F(x_{0}+t_{0}h)}{t_{1}-t_{0}}]+
		C\min\{\lambda, 1-\lambda\}\frac{\alpha(t-t_{0})}{t_{1}-t_{0}}k_{0}\in K.
\end{equation}
We have 

	\begin{description}
		\item {(i).}  If $\lambda\le 1-\lambda$, i.e.
		$t_{1}-t_{0}\le t-t_{0}$, then
		\begin{equation}
		\label{case_one}
			\min\{\lambda, 1-\lambda\}	\frac{\alpha(t-t_{0})}{t_{1}-t_{0}}=
			\frac{\alpha(t-t_{0})}{t-t_{0}}
		\end{equation}
		\item {(ii).} If $\lambda> 1-\lambda$, i.e.	$\frac{t_1-t_0}{t-t_0}> \frac{t-t_1}{t-t_0}$, then
		\begin{equation}
		\label{case_two} 
\min\{\lambda, 1-\lambda\}	\frac{\alpha(t-t_{0})}{t_{1}-t_{0}}=\frac{t-t_1}{t-t_0}\frac{\alpha(t-t_{0})}{t_{1}-t_{0}}<\frac{\alpha(t-t_{0})}{t-t_{0}}
		\end{equation}
		\end{description}
In both cases
			\begin{align*}
			[\frac{F(x_{0}+th)-F(x_{0}+t_{0}h)}{t-t_{0}}]-
		[\frac{F(x_{0}+t_{1}h)-F(x_{0}+t_{0}h)}{t_{1}-t_{0}}]+\\ C\frac{\alpha(t-t_{0})}{t-t_{0}}k_0-C\frac{\alpha(t_1-t_{0})}{t_{1}-t_{0}}k_0+C\frac{\alpha(t_1-t_{0})}{t_{1}-t_{0}}k_0\in K.
			\end{align*}
		\end{proof}

 If  $\mbox{int\,}K\neq \emptyset$, then any   strongly $\alpha(\cdot)$-$k_{0}$ paraconvex  mapping $F$   is strongly $\alpha(\cdot)$-$K$ paraconvex and for any $k\in K$ the $\alpha(\cdot)$-difference
quotients satisfy the formula \eqref{nier} with different constants $C$, and in general, one cannot find a single constant $C$ for all $0\neq k\in K$.

In the proposition below we investigate the boundedness of $\alpha(\cdot)$-difference quotient for strongly $\alpha(\cdot)$-$k$ paraconvex mappings.

\begin{proposition}
\label{wniosek1}
Let $X$ be a normed  space. Let $Y$ be a topological vector space and let $Y$
	be ordered by a closed convex pointed cone $K$.
	Let $F:X\rightarrow Y$ be   strongly $\alpha(\cdot)$-$k_{0}$  paraconvex  on a convex set $A\subset X$ with constant $C\ge 0$,  $k_{0}\in K\setminus\{0\}$.
	
	For  any  $x_{0}\in A$ and any $ h\in X$, 
	$\|h\|=1$ 
	such that $x_{0}+th\in A$ for all $t$ sufficiently small,
	 the $\alpha(\cdot)$-difference quotient mapping $\phi:[0,+\infty)\rightarrow Y$,
	\begin{equation}
	\label{phi1}
	\phi(t):=\frac{F(x_{0}+th)-F(x_{0})}{t}+
	C\frac{\alpha(t)}{t}k_0 \ %\text{for}\ t_{0}<t
	\end{equation}	
	is bounded from below in the sense that there is an element $a\in Y$ and $\delta >0$ such that
	\begin{equation}
	\label{ograniczenie1}
	\phi(t)-a\in K\ \ \text{for}\ \ 0<t<\delta.
	\end{equation}

\end{proposition}
	
\begin{proof}
%Let us take $\delta>0$ and $z>0$ as in the Fact \ref{fact1}. 
Let us take $t_0=-t$, $t_1=0$. From inclusion \eqref{nier}  we have
\begin{equation}
\label{eq_bound1}
\frac{F(x_{0}+th)-F(x_{0}-th)}{2t}+
	C\frac{\alpha(2t)}{2t}k_0- \frac{F(x_{0})-F(x_{0}-th)}{t}-
	C\frac{\alpha(t)}{t}k_0+C\frac{\alpha(t)}{t}k_0 \in K.\\
	\end{equation}
	Multiplying both sides by $2t>0$ we get
	$$
	F(x_{0}+th)-F(x_{0}-th)+C{\alpha(2t)}k_0-2F(x_0)+2F(x_0-th)\in K.
	$$
	By simple calculations we get
	$$\frac{F(x_{0}+th)-F(x_{0})}{t}+\frac{F(x_{0}-th)-F(x_{0})}{t}+2C\frac{\alpha(2t)}{2t}k_0\in K.
	$$
	Since  $\lim\limits_{t \rightarrow 0^+}\frac{\alpha(t)}{t}=0$, there exists $\delta >0$ such that  $2C\frac{\alpha(2t)}{2t} \le 1$ for $t\in (0, \delta)$. We  have 
	\begin{equation}
	\label{szacowanie1}
	\frac{F(x_{0}+th)-F(x_{0})}{t}+k_0\ge_K -\frac{F(x_{0}-th)-F(x_{0})}{t}.
	\end{equation}
	%Which shows that right side is bounded above and the left side is bounded below.
	
	%\end{proof}

Now, let us take $-1<-t<0$. We have
$$
x_0-th=t\underbrace{(x_0-h)}_{x_1}+(1-t)\underbrace{x_0}_{x_2}.
$$
From $\alpha()$-$k_0$ paraconvexity \eqref{def_para} for $\lambda:=t$ we get
$$
F(x_0-th)\le_K tF(x_0-h)+(1-t)F(x_0)+C\min\{t, 1-t\}\alpha(1)k_0
$$
By simple calculation we get

$$
-\frac{F(x_{0}-th)-F(x_{0})}{t}- F(x_0)+F(x_0-h)+C\frac{\min\{t, 1-t\}}{t}\alpha(1)k_0\in K.
$$
Since $\frac{\min\{t, 1-t\}}{t}=\frac{1-|2t-1|}{2t}$ and the fact that $\frac{1-|2t-1|}{2t}\le 1$ is bounded we get 
\begin{align*}
-\frac{F(x_{0}-th)-F(x_{0})}{t}- F(x_0)+F(x_0-h)+ C\alpha(1)k_0\in K% \\ +\underbrace{C\frac{\min\{t, 1-t\}}{t}\alpha(\|h\|)k_0-C\alpha(\|h\|)k_0}_{\le_K 0}\in K&.
\end{align*}
Hence, 
$$
-\frac{F(x_{0}-th)-F(x_{0})}{t}- F(x_0)+F(x_0-h)+C\alpha(1)k_0\in K.
$$
From \eqref{szacowanie1} we get
$$
\frac{F(x_{0}+th)-F(x_{0})}{t}-b\ge_K 0,
$$
where $b:=F(x_0)-F(x_0-h)-(C\alpha(1)+1)k_0.$
Finally,  
$$
\phi(t) - b \ge_K 0 \mbox{ for } 0<t<\delta.
$$
\end{proof}

\section{Main result}

The proof of the main theorem is based on the following lemma.

\begin{lemma}
	\label{pomocniczy}
	Let $Y$ be a Banach space. Let $K\subset Y$ be a closed convex normal cone. Let  $\Phi: \mathbb{R}_+ \rightarrow Y$ satisfy the following conditions
	\begin{itemize}
		\item[(i)] $\Phi(t) \in K$  for any  $t \in \mathbb{R}_+$,
		\item[(ii)] for  $0<t_1< t$ we have $\Phi(t)-\Phi(t_1) +  \frac{\alpha(t_1)}{t_1}k_0\in K$ for some $k_{0}\in K,$
		\item[(iii)] $\Phi(t)$ is weakly convergent to 0 when $t\rightarrow 0^+$
	\end{itemize}
	then $\|\Phi(t)\|\rightarrow 0$ when $t\rightarrow 0^+.$
\end{lemma}
\begin{proof}
	By contradiction, suppose that   $\|\Phi(t)\|\nrightarrow 0$ when $t\rightarrow 0^+$ and $(i)$ and $(ii)$ are satisfy. We will obtain a contradiction with (iii). By this, there is $\varepsilon >0$ such that for all $\delta>0$ one can find $0<t<\delta$ with $ \|\Phi(t)\| > \varepsilon.$ 
	In particular, for $\delta_n=\frac{1}{n}$  there exist $t_n\in (0,\frac{1}{n}), n\in \mathbb{N},$ such that 
	\begin{equation}
	\label{c}
	\|\Phi(t_n)\|>\varepsilon.
	\end{equation}
	Let $x\in A:=co(\Phi(t_n), n\in\mathbb{N}).$
	There are positive numbers $\lambda_1,\lambda_2,\dots,\lambda_m$ and $t_1,t_2\dots,t_m$ such that
	$x=\sum\limits_{i=1}^m \lambda_i \Phi(t_i),$ where $\sum_{i=1}^{m}\lambda_i=1.$ 
	There exist $N\in \mathbb{N}$  such that for all $n>N$ we have
	
	\begin{align*}
	\Phi(t_1)-\Phi(t_n) + \frac{\alpha(t_n)}{t_n}k_0\in &K, \\
	\Phi(t_2)-\Phi(t_n)+ \frac{\alpha(t_n)}{t_n}k_0\in &K,\\
	\vdots \\
	\Phi(t_m)-\Phi(t_n)+ \frac{\alpha(t_n)}{t_n}k_0\in &K.
	\end{align*}
	%Multiplying by $\lambda_1,\lambda_2,\dots,\lambda_m$ i dodaj?c do siebie (sto?ek $K$ jest wypuk?y) otrzymujemy
	We get
	$$
	x-\Phi(t_n)+ \frac{\alpha(t_n)}{t_n} k_0\in K \mbox{ for all } n>N.
	$$
	From the fact that $\Phi(t_n)\in K$ and  $K$ is normal there is some $c>0$
	such that 
	$\|\Phi(t_n)\|\le c \|x+ \frac{\alpha(t_n)}{t_n}k_0\|.$
	% Let us take $\beta:= \frac{\varepsilon}{c},$ 
	By \eqref{c}, we obtain $\|x+ \frac{\alpha(t_n)}{t_n}k_0\| > \beta:= \frac{\varepsilon}{c} \ \mbox{ for all } x\in A$ and
	$n>N.$
	
	We show  that
	$$
	\mathbb{B}_{\beta/2} \cap (A+k_0[0,s])=\emptyset
	$$ 
	for  $s>0$ satisfying
	$\frac{\alpha(t_n)}{t_n}\le s$.
	 To see this, take any $\ell\in (0,s],$ where $\mathbb{B}_r:=\{y\in Y: \|y\|\le r \}$. Since $\lim_{n\rightarrow +\infty}\frac{\alpha(t_n)}{t_n}=0$ there exists $n\in\mathbb{N}$ such that
	$$
	0\le_{K} x+ \frac{\alpha(t_n)}{t_n}k_0\le_{K} x+\ell k_{0}
	$$
	 By \eqref{c} and the normality of $K$,
	 	$$
	 	\beta/2< \| x+ \frac{\alpha(t_n)}{t_n}k_0\|\le\| x+\ell k_{0}\|
	 	$$
	From the Hahn-Banach theorem applied to $\mathbb{B}_{\beta/2}$ and $(A+k_0[0,s])$, 
	 there is a linear functional $y^*\in Y^*$ and $r>0$ such that
%	$$ 
%	y^*(x+ \frac{\alpha(t_n)}{t_n}k_0) > r  \ \mbox{ for all } x\in A+k_0[0,s].
%		$$
	
		$$ 
		y^*(x+ \ell k_0) > r  \ \mbox{ for all } x+\ell k_{0}\in A+k_0[0,s].
		$$
	In particular, $y^*(\Phi(t_n)+ \frac{\alpha(t_n)}{t_n}k_0)>r>0$,
	which contradicts (iii).
\end{proof}

We are in a position to prove our main result.

\begin{theorem}
\label{main}
		Let $X$ be a normed space. Let $Y$ be a weakly sequentially complete Banach space
		ordered by a closed convex normal cone $K$. 
		Let  $F:X\rightarrow Y$ be   strongly $\alpha(\cdot)$-$k_0$  paraconvex  on a convex set $A\subset X$ with constant $C\ge 0$,  $k_{0}\in K\setminus\{0\}$. Then the directional derivative
		$$
		F'(x_{0};h):=\lim_{t\rightarrow 0^{+}}\frac{F(x_{0}+th)-F(x_{0})}{t}
		$$
		of $F$ at  $x_{0}$ exists for any $x_0\in A$  and any direction $0\neq h\in X$, 
	$\|h\|=1$ 
	such that $x_{0}+th\in A$ for all $t$ sufficiently small.
		\end{theorem}
	
\begin{proof}
Let $x_0\in A$ and let $0\neq h\in X$, 
	$\|h\|=1$ 
	such that $x_{0}+th\in A$ for all $t$ sufficiently small.
 Let $t_n \downarrow 0$. For $t_{0}=0$  the $\alpha(\cdot)$-difference quotient % $\phi(t)$ defined
 % in Proposition \ref{mono}
  by  \eqref{phi0} takes the form
%Take $y^*\in K^*$, $t_n\downarrow 0$ and $t_{0}=0$. We have
   $$\phi(t_n)= \frac{F(x_{0}+t_nh)-F(x_{0})}{t_n}+
	C\frac{\alpha(t_n)}{t_n}k_0.
	$$ 
	%\color{blue}
	Let $y^*\in K^*$.
	By  \eqref{ograniczenie1}, sequence
	% in Proposition \ref{wniosek1} we get that
	 $a_n:=y^*(\phi(t_n))$, $n\in\mathbb{N}$ is bounded from below, i.e.
	\begin{equation*}
	\label{ogr2}
	a_n\ge a:=y^*(b) \  \text{for all } n \mbox{ sufficiently large and } b\in Y.
	\end{equation*}
	Let us take $\varepsilon > 0.$ There is $N$ such that
\begin{equation}
	\label{ogr2}
a_N< \underline{a} + \frac{\varepsilon}{2},
\end{equation}
where $\underline{a}:=\inf\{a_n: n\in \mathbb{N}\}.$
 Since $\{t_n\}$  is decreasing, from \eqref{nier} we get
 \begin{equation}
 \label{nier1}
 a_N-a_{n} + C\frac{\alpha(t_{n})}{t_{n}}y^*(k_0)\ge 0 \mbox{ for } n> N.
 \end{equation}
 Let $b_n:= C\frac{\alpha(t_{n})}{t_{n}}y^*(k_0)$. Since $b_n \rightarrow 0$ there is $N_1$ such that $b_n\le \frac{\varepsilon}{2}$ for $n>N_1.$

 From
 \eqref{ogr2} and \eqref{nier1} we get 
 $$
 \underline{a} - \varepsilon < \underline{a} \le a_n \le a_N +b_n \le \underline{a} +\frac{\varepsilon}{2} +b_n\le \underline{a}+\varepsilon \mbox{ for } n>\max\{N,N_1\}.
 $$
 Hence,  sequence $\{a_n\}$  is convergent and consequently 
 every sequence $\{y^*(\phi(t_n))\}$ is a  Cauchy for $y^*\in K^*$.
 
   	Let us take any  $h^* \in Y^*$. We show that the sequence $\{h^*(\phi(t_n))\}$ is Cauchy. From the fact that  $K$ is normal we have $Y^*=K^*-K^*$ and 
	%here are $g^*, q^* \in K^*$ such that
	 $h^*= g^*-q^*$ with 	$g^*, q^* \in K^*$.
	Since $\{g^*(\phi(t_n))\}$ and $\{q^*(\phi(t_n))\}$ are Cauchy sequences, 
	%for $\frac \varepsilon 2$ we have 
	there exist $N_1, N_2$ such that for $n,m > \bar{N}:=max(N_1, N_2)$ we have
	$$
	|g^*(\phi(t_n))- g^*(\phi(t_m))| \le \frac \varepsilon 2 \ \ \ \mbox {   and   } \ \ \  |q^*(\phi(t_n))- q^*(\phi(t_m))| \le \frac \varepsilon 2.
	$$
	
	For $n> \bar{N}$ we have
	\begin{equation}
	\begin{array}{c}
	|h^*(\phi(t_n))-h^*(\phi(t_m))|=|g^*(\phi(t_n))-q^*(\phi(t_n))-g^*(\phi(t_m))+q^*(\phi(t_m))|\le \frac \varepsilon 2 + \frac \varepsilon 2=\varepsilon.
	\end{array}
	\end{equation}
	We show that $\phi(t)$  weakly converges  when $t\rightarrow 0^+$ i.e.,
	%Since $Y$ is weakly sequentially complete,  
	there is an  $y_0 \in Y$ such that for arbitrary $t_n \downarrow 0$ we have
	$$
	%\label{granica_1}
	\lim\limits_{n\rightarrow \infty}   y^*(\phi(t_n)) = y^*(y_0) \ \mbox{ for any } y^* \in Y^*
	$$
	which is equivalent to 
	\begin{equation}
	\label{granica_1}
	\phi(t)\rightharpoonup y_0 \mbox{ when } t\rightarrow 0^+.
	\end{equation}
	
	Since $Y$ is weakly sequentially complete, we need only to show that $y_0$ is the same for all sequences $\{t_n\},$ $t_n \downarrow 0$. on the contrary,
	suppose that there are two different weak limits $y_0^1, y_0^2$ corresponding to  sequences $t_n^1$ and $t_n^2$, respectively.

	We can substract subsequences  $\{\bar{t}_{n}^2\}\subset \{t_n^2\}$ and $\{\bar{t}_{n}^1\}\subset \{t_n^1\}$ such that 
	$\bar{t}_{n}^2 \le t_n^1 \le \bar{t}_{n}^1.$ Correspondingly,
	$$
	 {y^*}(\phi(\bar{t}_{n}^2))\le {y^*}(\phi(t_{n}^1))\le y^*(\phi(\bar{t}_{n}^1))
	 $$
	 which proves that it must be $y_0^1=y_0^2$.
	%Since the weak limit is determined uniquely we get the contradiction.

	Now we show that the mapping $\Phi(t):=\phi(t)-y_0$ satisfies all the assumptions of Lemma  \ref{pomocniczy}.
	From \eqref{nier} and \eqref{granica_1} it is enough to show that  $\Phi(t)\in K$ for all $t\ge 0.$

	By contradiction, let us assume that  there is some $\bar{t}>0$ such that $\Phi(\bar{t})\notin K$. There exists $y^*\in K^*$
	such that
	\begin{equation}
	\label{cc}
	y^*(\Phi(\bar{t}))=y^*(\phi(\bar{t})-y_0)<0.
	\end{equation}
	From inclusion \eqref{nier} in Proposition \ref{mono} we have
	$$
	\phi(\bar{t})-y_0-\phi({t})+y_0  + C\frac{\alpha(t)}{t}k_0\in K \ \mbox{ for all } t\in (0,\bar{t}). $$
	In particular 
	$$
	y^*(\phi(\bar{t})-y_0) \ge y^*(\phi({t})-y_0-C\frac{\alpha(t)}{t}k_0) \mbox{ for all } t\in (0,\bar{t}).
	$$
	And by \eqref{cc} we get
	$$
	0>y^*( \phi(\bar{t})-y_0)\ge y^*( \phi({t})-y_0-C\frac{\alpha(t)}{t}k_0)\  \mbox{ for all } t\in (0,\bar{t}).
	$$
	
	Then by letting  $t\rightarrow 0^+$ we get the contradiction with \eqref{granica_1}.  
	%We get that the limit of $\phi(t_n)= \frac{F(x_{0}+t_nh)-F(x_{0})}{t}+
	%\beta\frac{\alpha(t_n)}{t_n}k$ exists.
	By Lemma \ref{pomocniczy}  $\Phi(t)$ tends to 0 when $t\rightarrow 0^+.$
	Since  $
		\lim_{t\rightarrow 0^{+}}\frac{\alpha(t)}{t}=0 
	$ 
	 we get 
	 $$
	 \lim\limits_{t\rightarrow 0^+} \frac{F(x_{0}+th)-F(x_{0})}{t}=y_{0}
	 $$
	  which completes the proof.

	\end{proof}
	\begin{remark}
	For $K$-convex mappings $F$ i.e. strongly $\alpha(\cdot)$-$K$ paracanovex mappings with constant $C=0$   Theorem \ref{main} can be found in \cite{valadier}. 
	\end{remark}

\end{document}